\documentclass{amsart}
\usepackage{amssymb}
\usepackage{amsthm}
\newtheorem{theor}{Theorem}[section] 
\newtheorem{prop}{Proposition} \newtheorem{cor}{Corollary}
\theoremstyle{definition} 
\newtheorem{ex}{Example}[section]
\theoremstyle{remark} 
\newcommand{\pn}{\par\noindent} \newcommand{\pmn}{\par\medskip\noindent}

\begin{document}
\title{On enumeration of tree-rooted planar cubic maps. II}
\author{Yury Kochetkov}
\date{}
\begin{abstract} In work \cite{Ko} tree-rooted planar cubic maps
with marked directed edge (not in this tree) were enumerated. The
number of such objects with $2n$ vertices is $C_{2n}\cdot
C_{n+1}$, where $C_k$ is Catalan number. In this work a marked
directed edge is not demanded, i.e. we enumerate tree-rooted
planar cubic maps. Formulas are more complex, of course, but not
significantly.
\end{abstract}
\email{yukochetkov@hse.ru, yuyukochetkov@gmail.com} \maketitle

\section{Introduction}
\pn \pn Plane triangulation is a planar map, where the perimeter
of each face is three. The corresponding dual graph is
\emph{cubic}, i.e. the degree of each vertex is three. A plane
triangulation will be called \emph{proper}, if each edge is
incident to exactly \emph{two} faces. Otherwise it will be called
\emph{improper}.
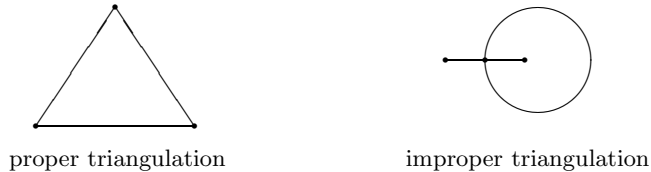
\begin{ex}
\[\begin{picture}(265,70) \put(15,5){\small proper triangulation}
\put(165,5){\small improper triangulation}
\put(25,20){\line(1,0){60}} \put(25,20){\circle*{2}}
\put(85,20){\circle*{2}} \put(55,65){\circle*{2}}
\put(25,20){\line(2,3){30}} \put(85,20){\line(-2,3){30}}
\put(215,45){\oval(40,40)} \put(180,45){\circle*{2}}
\put(195,45){\circle*{2}} \put(210,45){\circle*{2}}
\put(180,45){\line(1,0){30}}
\end{picture}\] The corresponding dual graphs are presented below:
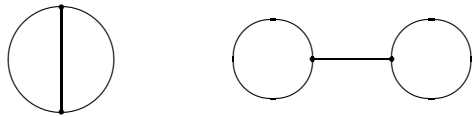
\[\begin{picture}(160,50) \put(20,25){\oval(40,40)} \put(20,5){\circle*{2}}
\put(20,45){\circle*{2}} \put(20,5){\line(0,1){40}} \put(60,22){¨}
\put(100,25){\oval(30,30)} \put(160,25){\oval(30,30)}
\put(115,25){\circle*{2}} \put(145,25){\circle*{2}}
\put(115,25){\line(1,0){30}}
\end{picture}\]
\end{ex} \pn A connected graph with marked directed edge will be called
\emph{edge-rooted}. Proper edge-rooted plane triangulations where
enumerated by Tutte in the work \cite{Tu}: the number of such
triangulations with $2n$ faces is
$$\frac{2\,(4n-3)!}{n!\,(3n-1)!}.$$ \pmn A combinatorial proof
of Tutte formula see in \cite{PS} (see also \cite{AP}). \pmn All
(proper and improper) plane edge-rooted triangulations were
enumerated in \cite{GJ} in the following way. Let $F_n$ be the
number of plane edge-rooted cubic graphs with $2n$ vertices, i.e.
the number of plane edge-rooted triangulations (proper and
improper) with $2n$ faces. Let us define numbers $f_n$,
$n\geqslant -1$, in the following way:
\begin{itemize}
    \item $f_{-1}=1/2$;
    \item $f_0=2$;
    \item $f_n=(3n+2)F_n$, $n>0$.
\end{itemize} Then
$$f_n=\frac{4(3n+2)}{n+1}\sum_{\scriptsize \begin{array}{c}
i\geqslant -1,\, j\geqslant -1\\ i+j=n-2\end{array}}
f(i)f(j).\eqno(1)$$ Thus $f_1=20, f_2=256$ and $F_1=4, F_2=32$.
However, we don't know any geometric or combinatorial proof of
this formula. \pmn
\begin{ex} It is easy to check the validity of this formula in
cases $n=1$ and $n=2$. There are two cubic maps with two vertices:
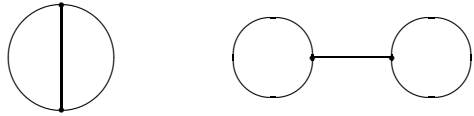
\[\begin{picture}(160,50) \put(20,25){\oval(40,40)} \put(20,5){\circle*{2}}
\put(20,45){\circle*{2}} \put(20,5){\line(0,1){40}} \put(60,22){¨}
\put(100,25){\oval(30,30)} \put(160,25){\oval(30,30)}
\put(115,25){\circle*{2}} \put(145,25){\circle*{2}}
\put(115,25){\line(1,0){30}} \end{picture}\] Orders of their
groups of automorphisms are $6$ and $2$, respectively. \pmn There
are six cubic maps with four vertices:
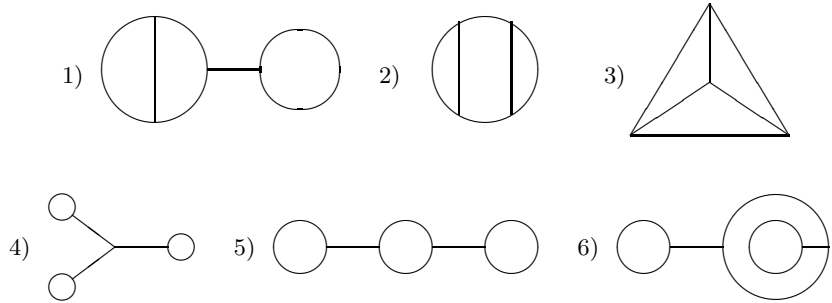
\[\begin{picture}(280,50) \put(40,25){\oval(40,40)} \put(40,5){\line(0,1){40}}
\put(60,25){\line(1,0){20}} \put(95,25){\oval(30,30)}
\put(5,20){\small 1)}

\put(165,25){\oval(40,40)} \put(155,8){\line(0,1){35}}
\put(175,8){\line(0,1){35}} \put(125,20){\small 2)}

\put(220,0){\line(1,0){60}} \put(220,0){\line(3,5){30}}
\put(220,0){\line(3,2){30}} \put(250,20){\line(3,-2){30}}
\put(250,20){\line(0,1){30}} \put(250,50){\line(3,-5){30}}
\put(210,20){\small 3)}
\end{picture}\]
\[\begin{picture}(310,50) \put(20,35){\circle{10}} \put(20,5){\circle{10}}
\put(40,20){\line(-4,3){16}} \put(40,20){\line(-4,-3){16}}
\put(40,20){\line(1,0){20}} \put(65,20){\circle{10}}
\put(0,17){\small 4)}

\put(110,20){\oval(20,20)} \put(120,20){\line(1,0){20}}
\put(150,20){\oval(20,20)} \put(160,20){\line(1,0){20}}
\put(190,20){\oval(20,20)} \put(85,17){\small 5)}

\put(240,20){\oval(20,20)} \put(250,20){\line(1,0){20}}
\put(290,20){\oval(40,40)} \put(290,20){\oval(20,20)}
\put(300,20){\line(1,0){10}} \put(215,17){\small 6)}
\end{picture}\] Orders of their
groups of automorphisms are $1$, $4$, $12$, $3$, $2$ and $2$,
respectively. \pmn The list of all $26$ cubic maps with six
vertices see in Appendix.
\end{ex} \pmn Let $G$ be a planar cubic map and $T$ --- a spanning tree. In
\cite{Ko} it was explained, how, using this data, to construct: a)
a polygon; b) its partition into triangles by non-intersecting
diagonals; c) a pairwise gluing of its edges that produces a curve
of genus 0 and its triangulation.

\begin{ex} Given a cubic map with a spanning tree (thick lines and curve) we
draw triangles around vertices:
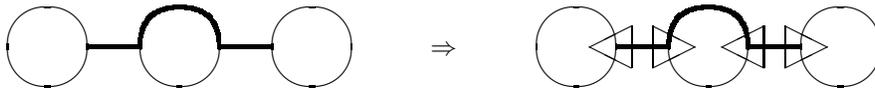
\[\begin{picture}(330,40) \put(15,20){\oval(30,30)}
\put(65,20){\oval(30,30)[b]} \put(115,20){\oval(30,30)}

\put(160,17){$\Rightarrow$}

\put(215,20){\oval(30,30)} \put(265,20){\oval(30,30)[b]}
\put(315,20){\oval(30,30)} \linethickness{0.5mm}
\put(230,20){\line(1,0){20}} \put(280,20){\line(1,0){20}}
\qbezier(250,20)(250,35)(265,35) \qbezier(280,20)(280,35)(265,35)
\put(30,20){\line(1,0){20}} \put(80,20){\line(1,0){20}}
\qbezier(50,20)(50,35)(65,35) \qbezier(80,20)(80,35)(65,35)

\thinlines \put(220,20){\line(2,1){16}}
\put(220,20){\line(2,-1){16}} \put(236,12){\line(0,1){16}}
\put(260,20){\line(-2,1){16}} \put(260,20){\line(-2,-1){16}}
\put(244,12){\line(0,1){16}} \put(270,20){\line(2,1){16}}
\put(270,20){\line(2,-1){16}} \put(286,12){\line(0,1){16}}
\put(310,20){\line(-2,1){16}} \put(310,20){\line(-2,-1){16}}
\put(294,12){\line(0,1){16}} \end{picture}\] After that we a)
identify those edges of adjacent triangles that intersect the same
edge of the spanning tree --- thus we obtain a polygon; b) label
identically those sides of polygon that are intersected  by the
same edge of the map; c) delete the map:
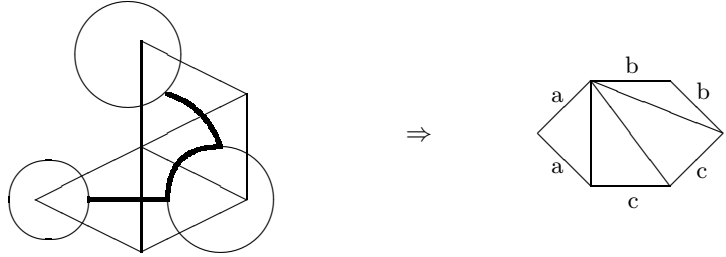
\[\begin{picture}(270,100) \put(10,25){\line(2,-1){40}}
\put(10,25){\line(2,1){80}} \put(50,5){\line(0,1){80}}
\put(50,5){\line(2,1){40}} \put(50,85){\line(2,-1){40}}
\put(90,25){\line(0,1){40}} \put(50,45){\line(2,-1){40}}
\linethickness{0.5mm} \put(30,25){\line(1,0){30}}
\qbezier(60,25)(60,45)(80,45) \qbezier(80,45)(75,60)(60,65)
\thinlines \put(15,25){\oval(30,30)} \put(45,80){\oval(40,40)}
\put(80,25){\oval(40,40)[b]} \put(80,25){\oval(40,40)[tr]}

\put(150,47){$\Rightarrow$}

\put(200,50){\line(1,1){20}} \put(200,50){\line(1,-1){20}}
\put(220,70){\line(0,-1){40}} \put(220,70){\line(3,-4){30}}
\put(220,70){\line(5,-2){50}} \put(220,70){\line(1,0){30}}
\put(220,30){\line(1,0){30}} \put(250,30){\line(1,1){20}}
\put(250,70){\line(1,-1){20}}

\put(205,62){\small a} \put(205,35){\small a} \put(233,73){\small
b} \put(234,22){\small c} \put(260,63){\small b}
\put(260,33){\small c} \end{picture}\] If we have a polygon: a)
triangulated by non-intersecting diagonals; b) its sides are
pairwise identified in such way, that this identification produces
a genus 0 curve, then, using these data, we can construct a planar
tree-based cubic map:
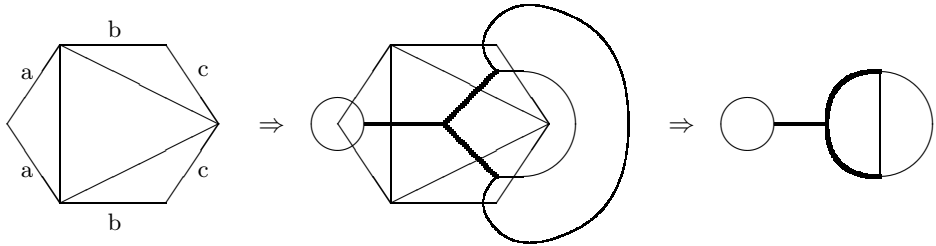
\[\begin{picture}(350,100) \put(0,50){\line(2,3){20}}
\put(0,50){\line(2,-3){20}} \put(20,20){\line(1,0){40}}
\put(20,20){\line(2,1){60}} \put(20,20){\line(0,1){60}}
\put(20,80){\line(1,0){40}} \put(20,80){\line(2,-1){60}}
\put(60,20){\line(2,3){20}} \put(60,80){\line(2,-3){20}}
\put(5,67){\small a} \put(5,30){\small a} \put(38,10){\small b}
\put(38,83){\small b} \put(72,68){\small c} \put(72,30){\small c}

\put(95,47){$\Rightarrow$}

\put(125,50){\line(2,3){20}} \put(125,50){\line(2,-3){20}}
\put(145,20){\line(1,0){40}} \put(145,20){\line(2,1){60}}
\put(145,20){\line(0,1){60}} \put(145,80){\line(1,0){40}}
\put(145,80){\line(2,-1){60}} \put(185,20){\line(2,3){20}}
\put(185,80){\line(2,-3){20}}

\linethickness{0.5mm}

\put(135,50){\line(1,0){30}} \qbezier(165,50)(175,60)(185,70)
\qbezier(165,50)(175,40)(185,30) \thinlines
\put(125,50){\oval(20,20)} \put(185,50){\oval(60,40)[r]}
\qbezier(185,70)(175,80)(185,90) \qbezier(185,30)(175,20)(185,10)
\qbezier(185,90)(195,100)(215,90) \qbezier(185,10)(195,0)(215,10)
\qbezier(215,90)(235,80)(235,50) \qbezier(215,10)(235,20)(235,50)

\put(250,47){$\Rightarrow$}

\put(280,50){\oval(20,20)} \put(330,50){\oval(40,40)[r]}
\put(330,30){\line(0,1){40}}

\linethickness{0.5mm}

\put(290,50){\line(1,0){20}} \qbezier(310,50)(310,70)(330,70)
\qbezier(310,50)(310,30)(330,30)
\end{picture}\]
\end{ex}  \pmn In what follows by $n$-polygon we understand a
polygon: a) with $2n$ sides; b) triangulated by non-intersecting
diagonals; b) with pairwise identified sides in such way, that
this identification produces a genus 0 curve. We see that the
enumeration of tree-rooted planar cubic maps is equivalent to the
enumeration of such objects. In work \cite{Ko} we additionally
mark a side of a polygon ( = we mark a directed edge in a cubic
map that does not belong to the spanning tree). The number of
$n$-polygons with a marked side is $C_n\cdot C_{2n-2}$, where
$C_k$ is the Catalan number. \pmn In this work we drop a marked
side, i.e. we enumerate $n$-polygons (Theorem 4.1).

\section{Triangulations and symmetries}
\pn
\begin{prop} Let $T$ be a $\mathbb{Z}_k$-symmetric triangulation
of a regular polygon $P$ with $n=km$ sides, i.e. $T$ is invariant
with respect to rotations on angles $2\pi\,i/k$, $i=1,\ldots,k-1$.
Then $k=2$ or $k=3$. \end{prop}

\begin{proof} Let $R$ be the rotation on angle $2\pi/k$. As
numbers of triangles in the triangulation $T$ is $n-2$, then there
exists a triangle $\triangle$ in $T$ whose obtuse angle is an
angle of $P$. As triangulation $T$ is $\mathbb{Z}_k$-symmetric
then there $k$ copies of triangle $\triangle$ in $T$ and their
configuration is symmetric with respect to rotation $R$. \pmn If
we delete these triangles from polygon $P$, then we will obtain
new polygon $P_1$ with $k(m-1)$ sides with triangulation $T_1$,
which is also $\mathbb{Z}_k$-symmetric. \pmn If we continue the
process, then after $m-1$ steps we will obtain a polygon with $k$
sides and with $\mathbb{Z}_k$-symmetric triangulation. But it is
possible only when $k=2$ or $k=3$. \end{proof}

\begin{cor} Let $P$ be a regular polygon with $2n$
sides, one of them --- marked. Then the number of its
$\mathbb{Z}_2$-symmetric triangulations is $n\cdot C_{n-1}$.
\end{cor}

\begin{proof} Indeed, there are $n$ ways to choose a big diagonal
and $C_{n-1}$ ways to triangulate a half of the polygon.
\end{proof}

\begin{cor} Let $P$ be a regular polygon with $3n$
sides, one of them --- marked. Then the number of its
$\mathbb{Z}_3$-symmetric triangulations is $n\cdot C_{n-1}$.
\end{cor}

\begin{proof} Indeed, there are $n$ ways to choose a
$\mathbb{Z}_3$-symmetric triangle in $P$. The complement of this
triangle consists of three symmetrically positioned polygons with
$n+1$ sides each. There is $C_{n-1}$ ways to triangulate one of
them. \end{proof}

\section{Glueings and symmetries}
\pn
\begin{prop}  Let $P$ be a regular polygon with $2n$
sides, one of them --- marked and let $d_n$ the number of such
$\mathbb{Z}_2$-symmetric pairwise glueings of its sides, that
produce a genus zero curve. Then
$$d_n=\begin{cases}\binom{2k}{k}, \;\;\;\text{ if $n=2k$,}\\ \frac 12\,
\binom{2k}{k},\text{ if $n=2k-1$.} \end{cases}$$
\end{prop}.
\pmn \emph{Proof.} Let us consider a $\mathbb{Z}_2$-symmetric zero
genus glueing of $P$. Let $AB$ be the marked side, $CD$ --- the
opposite side, $KL$ --- the glueing partner of $AB$ and $MN$ ---
the glueing partner of $CD$ (sides $KL$ and $MN$ are opposite). In
the counterclockwise going around of $P$ one can meet the side
$KL$ before $CD$ or after it. We will consider the first case and
then multiply by two. The number of sides between $AB$ and $KL$ is
even $=2k$. We can choose any zero genus pairwise glueing of sides
between $AB$ and $KL$ (there are $C_k$ of them), made the
symmetric glueing of sides between $CD$ and $MN$ and choose a
$\mathbb{Z}_2$-symmetric glueing of remaining sides (there are
$d_{n-k-2}$ of them). Thus,
$$d_n=2d_{n-2}+2d_{n-4}C_1+2d_{n-6}C_2+\ldots+2d_2C_{(n-4)/2}+2C_{(n-2)/2},$$
if $n$ is even, and
$$d_n=2d_{n-2}+2d_{n-4}C_1+\ldots+2d_1C_{(n-3)/2}+C_{(n-1)/2},$$
if $n$ is odd. \pmn It means that
$$\left(\frac{1}{2t}+1+d_1t+d_2t^2+\ldots\right)\times
(2t^2+C_1t^4+2C_2t^6+\ldots)=d_1t+d_2t^2+d_3t^3+\ldots$$ Let
$$D(t)=\frac{1}{2t}+1+d_1 t+d_2t^2+\ldots$$ and let $C(t)$ be the
generating function for Catalan numbers, then
$$2t^2C(t^2)D(t)=D(t)-1-1/2t,\text{ or }D(t)(1-2t^2C(t^2))=1+1/2t.$$ As
$C(t)=(1-\sqrt{1-4t}\,\,)/2t$, then
$$D(t)=\left(1+\frac{1}{2t}\right)(1-4t^2)^{-1/2}=
\left(1+\frac{1}{2t}\right)\sum_{i=0}^\infty \binom{2i}{i}
t^{2i}.\qquad\Box$$

\begin{prop}  Let $P$ be a regular polygon with $6n$
sides, one of them --- marked and let $e_n$ the number of such
$\mathbb{Z}_3$-symmetric pairwise glueings of its sides, that
produce a genus zero curve. Then $e_n=\binom{2n}{n}$.\end{prop}

\begin{proof} The same reasoning gives the relation
$$2(e_{n-1}+e_{n-1}C_1+\ldots+e_1C_{n-2}+C_{n-1})=e_n,$$ i.e.
$2tC(t)E(t)=E(t)-1$, where $E(t)=1+e_1t+e_2t^2+\ldots$. Thus,
$$E(t)=\frac{1}{1-2tE(t)}=(1-4t)^{-1/2}=\sum_{i=0}^\infty
\binom{2i}{i}t^i.$$\end{proof}

\section{Enumeration formulas}
\pn It remains to apply Burnside formula.

\begin{theor} Let $t_n$ be the number of tree-rooted planar cubic
maps with $2n$ vertices. Then
$$(2n+2)t_n=\begin{cases} C_{n+1}C_{2n}+\frac 12\,(n+1)\,C_n\cdot\binom{2k}{k},
\text{ if $n+1=2k-1$ and $n+1\neq 3l$},\\
C_{n+1}C_{2n}+(n+1)\,C_n\cdot \binom{2k}{k}, \quad\text{if
$n+1=2k$
and $n+1\neq 3l$},\\
C_{n+1}C_{2n}+4(2k-1)\,C_{4k-3}\cdot\binom{4k-2}{2k-1}+\frac
12\,(n+1)\,C_n\cdot
\binom{6k-2}{3k-1},\\ \hspace{47mm}\text{ if $n+1=3(2k-1)$},\\
C_{n+1}C_{2n}+8k\,C_{4k-1}\cdot\binom{4k}{2k}+(n+1)\,C_n\cdot\binom{6k}{3k},
\text{ if $n+1=6k$}. \end{cases}$$
\end{theor}

\begin{ex} If $2n=4$ then there are $16$ tree-rooted planar cubic
maps. Here are they:
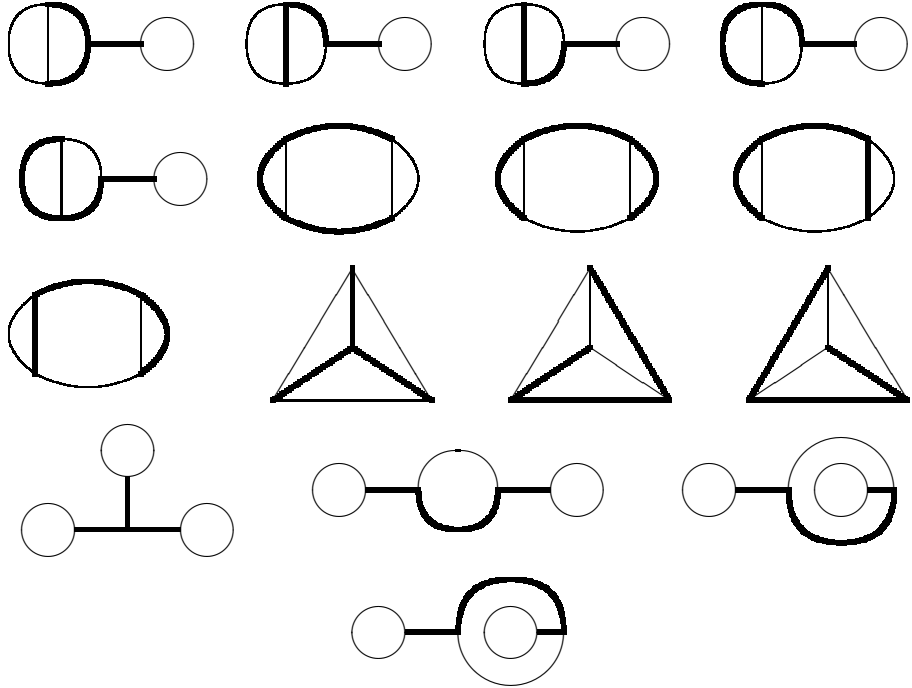
\[\begin{picture}(340,40) \qbezier(0,20)(0,35)(15,35)
\qbezier(0,20)(0,5)(15,5) \put(15,5){\line(0,1){30}}
\put(60,20){\oval(20,20)} \linethickness{0.6mm}
\qbezier(15,35)(30,35)(30,20) \qbezier(15,5)(30,5)(30,20)
\qbezier(30,20)(40,20)(50,20)

\thinlines \qbezier(90,20)(90,35)(105,35)
\qbezier(90,20)(90,5)(105,5) \qbezier(105,5)(120,5)(120,20)
\put(150,20){\oval(20,20)} \linethickness{0.6mm}
\qbezier(105,35)(120,35)(120,20) \qbezier(120,20)(130,20)(140,20)
\qbezier(105,5)(105,20)(105,35)

\thinlines \qbezier(180,20)(180,35)(195,35)
\qbezier(180,20)(180,5)(195,5) \qbezier(195,35)(210,35)(210,20)
\put(240,20){\oval(20,20)} \linethickness{0.6mm}
\qbezier(195,5)(210,5)(210,20) \qbezier(210,20)(220,20)(230,20)
\qbezier(195,5)(195,20)(195,35)

\thinlines \qbezier(285,5)(300,5)(300,20)
\put(285,5){\line(0,1){30}} \put(330,20){\oval(20,20)}
\linethickness{0.6mm} \qbezier(270,20)(270,35)(285,35)
\qbezier(285,35)(300,35)(300,20) \qbezier(300,20)(310,20)(320,20)
\qbezier(270,20)(270,5)(285,5)
\end{picture}\]
\[\begin{picture}(330,50)  \put(15,10){\line(0,1){30}}
\put(60,25){\oval(20,20)} \qbezier(15,40)(30,40)(30,25)
\linethickness{0.6mm} \qbezier(0,25)(0,40)(15,40)
\qbezier(0,25)(0,10)(15,10) \qbezier(15,10)(30,10)(30,25)
\qbezier(30,25)(40,25)(50,25)

\thinlines \put(100,10){\line(0,1){30}}
\put(140,10){\line(0,1){30}} \qbezier(140,10)(160,25)(140,40)
\linethickness{0.6mm} \qbezier(100,10)(80,25)(100,40)
\qbezier(100,40)(120,50)(140,40) \qbezier(100,10)(120,0)(140,10)

\thinlines \put(190,10){\line(0,1){30}}
\qbezier(190,10)(210,0)(230,10) \put(230,10){\line(0,1){30}}
\linethickness{0.6mm} \qbezier(190,10)(170,25)(190,40)
\qbezier(190,40)(210,50)(230,40) \qbezier(230,10)(250,25)(230,40)

\thinlines \put(280,10){\line(0,1){30}}
\qbezier(320,10)(340,25)(320,40) \qbezier(280,10)(300,0)(320,10)
\linethickness{0.6mm} \qbezier(280,10)(260,25)(280,40)
\qbezier(280,40)(300,50)(320,40) \qbezier(320,10)(320,25)(320,40)
\end{picture}\]
\[\begin{picture}(340,60) \put(50,15){\line(0,1){30}}
\qbezier(10,15)(-10,30)(10,45) \qbezier(10,15)(30,5)(50,15)
\linethickness{0.6mm} \qbezier(10,45)(30,55)(50,45)
\qbezier(50,15)(70,30)(50,45) \qbezier(10,15)(10,30)(10,45)

\thinlines \put(100,5){\line(1,0){60}} \put(100,5){\line(3,5){30}}
\put(160,5){\line(-3,5){30}} \linethickness{0.6mm}
\qbezier(100,5)(115,15)(130,25) \qbezier(160,5)(145,15)(130,25)
\qbezier(130,25)(130,40)(130,55)

\thinlines \put(190,5){\line(3,5){30}}
\put(250,5){\line(-3,2){30}} \put(220,25){\line(0,1){30}}
\linethickness{0.6mm} \qbezier(190,5)(220,5)(250,5)
\qbezier(190,5)(205,15)(220,25) \qbezier(250,5)(235,30)(220,55)

\thinlines \put(340,5){\line(-3,5){30}}
\put(280,5){\line(3,2){30}} \put(310,25){\line(0,1){30}}
\linethickness{0.6mm} \qbezier(280,5)(310,5)(340,5)
\qbezier(280,5)(295,30)(310,55) \qbezier(340,5)(325,15)(310,25)
\end{picture}\]
\[\begin{picture}(330,50) \put(10,10){\oval(20,20)}
\put(70,10){\oval(20,20)} \put(40,40){\oval(20,20)}
\linethickness{0.6mm} \put(20,10){\line(1,0){40}}
\put(40,10){\line(0,1){20}}

\thinlines \put(120,25){\oval(20,20)}
\put(165,25){\oval(30,30)[t]} \put(210,25){\oval(20,20)}
\linethickness{0.6mm} \put(130,25){\line(1,0){20}}
\put(180,25){\line(1,0){20}} \qbezier(150,25)(150,10)(165,10)
\qbezier(165,10)(180,10)(180,25)

\thinlines \put(260,25){\oval(20,20)}
\put(310,25){\oval(40,40)[t]} \put(310,25){\oval(20,20)}
\linethickness{0.6mm} \put(270,25){\line(1,0){20}}
\put(320,25){\line(1,0){10}} \qbezier(290,25)(290,5)(310,5)
\qbezier(310,5)(330,5)(330,25) \end{picture}\]
\[\begin{picture}(80,50) \put(10,25){\oval(20,20)}
\put(60,25){\oval(40,40)[b]} \put(60,25){\oval(20,20)}
\linethickness{0.6mm} \put(20,25){\line(1,0){20}}
\put(70,25){\line(1,0){10}} \qbezier(40,25)(40,45)(60,45)
\qbezier(60,45)(80,45)(80,25) \end{picture}\]
\end{ex}

\section{Appendix. Planar cubic maps with six vertices}
\pn We give the list of pairwise non isomorphic planar cubic maps
with 6 vertices --- there are $26$ of them. The group of
automorphisms of a map (if this group is non trivial) is also
presented.
\[\begin{picture}(340,50) \put(0,22){\small 1)}
\put(40,25){\oval(40,40)} \put(40,25){\oval(22,22)}
\put(20,25){\line(1,0){9}} \put(40,25){\oval(10,10)}
\put(45,25){\line(1,0){6}} \put(60,25){\line(1,0){10}}
\put(77,25){\oval(14,14)} \put(92,22){\small $\mathbb{Z}_2$}

\put(108,22){\small 2)} \multiput(130,25)(24,0){4}{\oval(14,14)}
\multiput(137,25)(24,0){3}{\line(1,0){10}} \put(220,22){\small
$\mathbb{Z}_2$}

\put(240,22){\small 3)} \multiput(262,25)(24,0){2}{\oval(14,14)}
\multiput(269,25)(24,0){2}{\line(1,0){10}}
\put(318,25){\oval(30,30)} \put(318,25){\oval(14,14)}
\put(325,25){\line(1,0){8}} \end{picture}\]

\[\begin{picture}(340,35) \put(0,18){\small 4)}
\put(20,20){\oval(10,10)} \put(50,20){\oval(10,10)}
\put(25,20){\line(1,0){20}} \put(35,20){\line(0,1){10}}
\put(35,35){\oval(10,10)} \put(55,20){\line(1,0){10}}
\put(70,20){\oval(10,10)}

\put(90,18){\small 5)} \put(110,20){\oval(10,10)}
\put(150,20){\oval(30,30)} \put(115,20){\line(1,0){20}}
\put(125,20){\line(0,1){10}} \put(125,35){\oval(10,10)}
\put(150,20){\oval(10,10)} \put(155,20){\line(1,0){10}}

\put(180,18){\small 6)} \multiput(200,20)(24,0){3}{\oval(14,14)}
\multiput(207,20)(24,0){2}{\line(1,0){10}}
\put(248,13){\line(0,1){14}}

\put(265,18){\small 7)} \put(290,20){\oval(20,20)[l]}
\multiput(290,10)(10,0){3}{\line(0,1){20}}
\multiput(290,10)(0,20){2}{\line(1,0){20}}
\put(310,20){\oval(20,20)[r]} \put(330,18){\small $\mathbb{Z}_2$}
\end{picture}\]

\[\begin{picture}(340,40) \put(0,18){\small 8)}
\put(25,20){\oval(14,14)} \put(32,20){\line(1,0){10}}
\put(25,13){\line(0,1){14}} \put(57,20){\oval(30,30)}
\put(57,20){\oval(10,10)} \put(62,20){\line(1,0){10}}

\put(85,18){\small 9)} \put(110,20){\oval(14,14)}
\put(117,20){\line(1,0){20}} \put(110,13){\line(0,1){14}}
\put(142,20){\oval(10,10)} \put(127,20){\line(0,1){10}}
\put(127,35){\oval(10,10)}

\put(155,18){\small 10)} \multiput(180,20)(24,0){3}{\oval(14,14)}
\multiput(187,20)(24,0){2}{\line(1,0){10}}
\put(204,13){\line(0,1){14}} \put(245,18){\small $\mathbb{Z}_2$}

\put(265,18){\small 11)} \put(290,20){\oval(14,14)}
\put(290,13){\line(0,1){14}} \put(297,20){\line(1,0){10}}
\put(314,20){\oval(14,14)} \put(314,13){\line(0,1){14}}
\put(330,18){\small $\mathbb{Z}_2$}
\end{picture}\]

\[\begin{picture}(340,40) \put(0,18){\small 12)}
\put(25,20){\oval(10,10)} \put(30,20){\line(1,0){40}}
\put(40,20){\line(0,1){10}} \put(60,20){\line(0,1){10}}
\put(75,20){\oval(10,10)} \put(40,35){\oval(10,10)}
\put(60,35){\oval(10,10)} \put(90,18){\small $\mathbb{Z}_2$}

\put(110,18){\small 13)} \multiput(135,20)(20,0){3}{\oval(10,10)}
\multiput(140,20)(20,0){2}{\line(1,0){10}}
\put(155,25){\line(0,1){10}} \put(155,40){\oval(10,10)}
\put(190,18){\small $\mathbb{Z}_3$}

\put(210,18){\small 14)} \put(250,20){\oval(40,30)[l]}
\put(230,20){\line(1,0){6}} \put(241,20){\oval(10,10)}
\put(250,5){\line(0,1){30}} \put(250,20){\oval(20,30)[r]}
\put(260,20){\line(1,0){10}} \put(275,20){\oval(10,10)}

\put(290,18){\small 15)} \put(325,15){\oval(20,20)[b]}
\put(315,15){\line(0,1){10}} \put(335,15){\line(0,1){10}}
\put(315,15){\line(1,0){20}} \put(315,25){\line(1,0){20}}
\put(325,25){\oval(20,20)[t]} \put(325,35){\line(0,1){5}}
\put(325,45){\oval(10,10)}
\end{picture}\]

\[\begin{picture}(340,40) \put(0,18){\small 16)}
\put(40,20){\oval(40,30)[l]} \put(40,5){\line(0,1){30}}
\put(40,20){\line(-1,0){6}} \put(29,20){\oval(10,10)}
\put(40,20){\oval(20,30)[r]} \put(50,20){\line(1,0){10}}
\put(65,20){\oval(10,10)} \put(80,16){\small $\mathbb{Z}_2$}

\put(100,18){\small 17)} \put(135,20){\oval(30,30)}
\put(135,20){\oval(10,10)} \put(135,5){\line(0,1){10}}
\put(135,25){\line(0,1){10}} \put(150,20){\line(1,0){10}}
\put(165,20){\oval(10,10)}

\put(183,18){\small 18)} \put(208,20){\oval(10,10)}
\put(213,20){\line(1,0){8}} \put(236,20){\oval(30,30)}
\put(236,5){\line(0,1){10}} \put(236,20){\oval(10,10)}
\put(251,20){\line(1,0){8}} \put(264,20){\oval(10,10)}

\put(280,18){\small 19)} \put(310,20){\oval(20,20)}
\put(320,20){\line(1,0){10}} \put(335,20){\oval(10,10)}
\put(310,10){\line(0,1){20}} \put(300,20){\line(1,0){10}}
\end{picture}\]

\[\begin{picture}(340,40) \put(0,18){\small 20)}
\put(45,15){\oval(20,20)[b]} \put(35,15){\line(1,0){20}}
\put(35,15){\line(0,1){20}} \put(55,15){\line(0,1){20}}
\put(35,25){\line(-1,0){5}} \put(25,25){\oval(10,10)}
\put(55,25){\line(1,0){5}} \put(65,25){\oval(10,10)}
\put(45,35){\oval(20,20)[t]}

\put(92,18){\small 21)} \put(137,15){\oval(20,20)[b]}
\put(127,15){\line(1,0){20}} \put(127,15){\line(0,1){20}}
\put(147,15){\line(0,1){20}} \put(127,25){\line(-1,0){5}}
\put(117,25){\oval(10,10)} \put(147,25){\line(-1,0){5}}
\put(137,25){\oval(10,10)} \put(137,35){\oval(20,20)[t]}
\put(157,18){\small $\mathbb{Z}_2$}

\put(189,18){\small 22)} \put(219,15){\oval(20,20)[b]}
\put(209,15){\line(1,0){20}} \put(209,15){\line(0,1){20}}
\put(229,15){\line(0,1){20}} \put(209,25){\line(1,0){5}}
\put(219,25){\oval(10,10)} \put(229,25){\line(1,0){5}}
\put(239,25){\oval(10,10)} \put(219,35){\oval(20,20)[t]}
\put(254,18){\small $\mathbb{Z}_2$}

\put(285,18){\small 23)} \put(330,15){\oval(20,20)[b]}
\put(320,15){\line(1,0){20}} \put(320,15){\line(0,1){20}}
\put(340,15){\line(0,1){20}} \put(320,25){\line(-1,0){5}}
\put(310,25){\oval(10,10)} \put(320,35){\line(1,0){20}}
\put(330,35){\oval(20,20)[t]} \end{picture}\]

\[\begin{picture}(300,30) \put(0,13){\small 24)}
\put(35,15){\oval(30,30)} \put(35,15){\oval(10,10)}
\put(20,15){\line(1,0){10}} \put(35,20){\line(0,1){10}}
\put(40,15){\line(1,0){10}} \put(60,13){\small $S_3$}

\put(80,13){\small 25)} \put(110,15){\oval(20,20)[l]}
\put(110,5){\line(0,1){20}} \put(110,5){\line(1,0){20}}
\put(110,25){\line(1,0){20}} \put(130,15){\oval(20,20)[r]}
\put(130,5){\line(0,1){20}} \put(130,15){\line(1,0){10}}
\put(150,13){\small $\mathbb{Z}_2$}

\put(175,13){\small 26)} \put(210,15){\oval(20,20)[l]}
\put(210,5){\line(0,1){20}} \put(210,5){\line(1,0){40}}
\put(210,25){\line(1,0){40}} \put(250,5){\line(0,1){20}}
\put(250,15){\oval(20,20)[r]} \put(230,25){\oval(20,20)[t]}
\put(270,13){\small $S_3$} \end{picture}\]

\vspace{5mm}
\end{document}